\documentclass[11pt]{scrartcl}
\usepackage{amsfonts,amssymb,amsthm,amsmath,booktabs}
\usepackage{color,graphicx}
\usepackage[latin1]{inputenc}

\ifx\pdftexversion\undefined
\usepackage{nohyperref}
\else
\usepackage[colorlinks=true,linkcolor={blue},citecolor={blue},urlcolor={blue},
  pdfauthor={Thomas Apel, Serge Nicaise, Johannes Pfefferer},
  pdfstartview={Fit}]{hyperref}
\fi

\newcommand{\abssec}[1]{\noindent\normalsize {\bfseries #1\quad }\ignorespaces}
\renewenvironment{abstract}{\abssec{Abstract}}{\par\vspace{.1in}}
\newenvironment{keywords}{\abssec{Key Words}}{\par\vspace{.1in}}
\newenvironment{AMSMOS}{\abssec{AMS subject
  classification}}{\par\vspace{.1in}}

\theoremstyle{plain}
\newtheorem{theorem}{Theorem}[section]

\newtheorem{lemma}[theorem]{Lemma}

\theoremstyle{definition}
\newtheorem{remark}[theorem]{Remark}
\newtheorem{example}[theorem]{Example}

\numberwithin{equation}{section}

\def\R{\mathbb{R}}

\renewcommand{\imath}{\mathrm{i}\,}

\definecolor{darkgreen}{rgb}{0.0, 0.5, 0.3}

\begin{document}

\title{\LARGE Regularity of the solution of the \\
  scalar Signorini problem in polygonal domains\thanks{Partially
    funded by the Deutsche Forschungsgemeinschaft (DFG, German
    Research Foundation) -- Projektnummer 188264188/GRK1754.}}

\author{Thomas Apel\thanks{\texttt{thomas.apel@unibw.de}, Institut
    f\"ur Mathematik und Computergest\"utzte Simulation, Universit\"at
    der Bundeswehr M\"unchen, D-85579 Neubiberg, Germany} \and Serge
  Nicaise\thanks{\texttt{serge.nicaise@uphf.fr}, Universit\'e
    Polytechnique Hauts-de-France, LAMAV, FR CNRS 2956, F-59313 -
    Valenciennes Cedex 9, France}}
\maketitle

\begin{abstract}
  The Signorini problem for the Laplace operator is considered in a
  general polygonal domain. It is proved that the coincidence set
  consists of a finite number of boundary parts plus isolated points.
  The regularity of the solution is described. In particular, we
    show that the leading singularity is in general
    $r_i^{\pi/(2\alpha_i)}$ at transition points of Signorini to
    Dirichlet or Neumann conditions but $r_i^{\pi/\alpha_i}$ at kinks
    of the Signorini boundary, with $\alpha_i$ being the internal
    angle of the domain at these critical points.
\end{abstract}


\begin{AMSMOS}
  35B65; 
      49N60 
\end{AMSMOS}

\section{Introduction}

In this paper we consider the Signorini problem
\begin{align} \label{eq:bvp}
  -\Delta y &= 0 \quad\text{in }\Omega, \\
  y &= 0 \quad\text{on }\Gamma_D,\\
   \partial_n y&=0 \quad\text{on }\Gamma_N,\\
  \partial_n y&=u \quad\text{on }\Gamma_U,\\
  y\geq 0,\ \partial_n y\geq 0, \ y \partial_n y&=0 \quad\text{on }\Gamma_S,
 \label{eq:signorinibc}
\end{align}
with a boundary datum $u\in L^2(\Gamma_U)$. We assume that the
mutually disjoint,  relatively open sets $\Gamma_D$, $\Gamma_N$, $\Gamma_U$,
and $\Gamma_S$ satisfy
\begin{align}\label{boundaryparts}
  \bar \Gamma_D\cup \bar \Gamma_N\cup \bar \Gamma_U\cup \bar \Gamma_S
  =\Gamma=\partial\Omega, \quad
  \bar\Gamma_S\cap\bar\Gamma_U=\emptyset, \quad \Gamma_D\ne \emptyset
\end{align} 
with $\Gamma$ being the boundary of the bounded polygonal domain
$\Omega\subset\R^2$.  The boundary parts $\Gamma_N$ and $\Gamma_U$ are
distinguished because of the second assumption in
\eqref{boundaryparts}. The condition $\Gamma_D\ne \emptyset$ is
assumed to obtain a unique solution.  The notation and our interest in
the problem comes from an optimal control problem where $y$ is the
state variable and $u$ is the control variable. 

Problem \eqref{eq:bvp}--\eqref{eq:signorinibc} is sometimes considered
as the scalar version of the more important Signorini problem for the
Lam\'e equations (``linear elasticity with unilateral boundary
condition'') but it has its own application describing a steady-state
fluid mechanics problem in media with a semi-permeable boundary, see
\cite[Section 1.1.1]{GlowinskiLionsTremolieres81}.

Let $C=\{c_i\}_{i=1}^n$ be the set of critical boundary points, namely
all points where the type of the boundary condition changes, that is
$\Gamma\setminus(\Gamma_D\cup\Gamma_N\cup\Gamma_U\cup\Gamma_S)$,
and all corners of the domain.  Br\'{e}zis \cite{Brezis:71} (see
also \cite{ Fichera:71} for the elasticity system) showed for the
inhomogeneous equation in smooth domains with purely Signorini
boundary condition that the solution is $H^2$-regular, Grisvard and
Iooss \cite{Grisvard-Iooss:76} extended this result to the case of
convex domains. Moussaoui and Khodja \cite{MoussaouiKhodja1992} showed
$C^{1,\lambda}$-regularity away from $C$ for $\lambda<\frac12$, see
also Theorem \ref{t:regy}; they further discussed possible singular
behavior near the critical points, see also Theorem \ref{t:reg}. This
last description suggests the $H^t$-regularity with $t\in (2,5/2)$ of
the solution near $\Gamma_S$.  Consequently some authors
\cite{Belhachmi-Belgacem:03,Drouet-Hild:15,Wohlmuthetall:12} assume
such a regularity without a complete proof, and use it for their numerical
analysis of the problem.  However, for the analysis of the behavior
near the extremal points of $\Gamma_S$ and for sharper regularity
results one needs that the coincidence set
\begin{align}\label{def:GammaC}
  \Gamma_C=\{x\in\Gamma_S:y(x)=0\}
\end{align}
consists of a finite number of connected boundary parts
(``intervals'') plus isolated points. Otherwise the set of endpoints
of the coincidence set (the set of points where the condition $u=0$
changes to $u>0$) could possess accumulation points while the analysis
of the regularity near such points (or near corners of the domain)
assumes the existence of a $\delta$-neighborhood where the type of the
boundary condition does not change.  As a consequence there  are publications where the
structure of the coincidence set is formulated as an assumption, see,
e.g., \cite[Condition (A)]{ChristofHaubner18}.

One important result of our paper is the proof of this proposition in
Section \ref{sec:coincidence}.  Such a result was previously obtained
for the Signorini problem with the Lam\'e equations by Kinderlehrer
 in \cite{kinderlehrer:81Pisa,kinderlehrer:81JMPA} 
under the assumptions
\begin{itemize}
\item that the boundary of $\Omega$ is flat in a neighborhood of 
  $\Gamma_S$, more precisely that
  \[
  \Gamma_S=(-c,c)\times \{0\}\subset 
  \tilde \Gamma=(-\tilde c,\tilde c)\times \{0\} \subset \partial\Omega
  \]
  for some positive constants $c$ and $\tilde c$ such that $c<\tilde
  c$, and
\item that the part $\tilde \Gamma\setminus\Gamma_S$ is included into
  $\Gamma_N$. 
\end{itemize}
While the transfer to the Laplace equation and to the case that
$\tilde \Gamma\setminus\Gamma_C\subset\Gamma_D$ can be done with
similar ideas, the avoidance of the the first assumption above is not
straightforward. The main tool for our proof is a special conformal
mapping which preserves the differential operator in \eqref{eq:bvp}
and the normal derivative. It is not clear how to analyze other
equations or a domain with curved boundary.
For simplicity of presentation we assumed that the differential
equation in \eqref{eq:bvp} and the gap function in
\eqref{eq:signorinibc} are homogeneous. We admit that we cannot treat
the general case but we discuss  examples in Remark \ref{rem:f}.

With the knowledge of the structure of the coincidence set one can
analyze the regularity of the solution, see, e.\,g., the already
mentioned paper \cite{MoussaouiKhodja1992}  by Moussaoui and Khodja for
results in H\"older spaces. We discuss the regularity in Sobolev
spaces in Section \ref{sec:reg} where we use a form which is useful
for our forthcoming numerical analysis.

\section{\label{sec:coincidence}The coincidence set}


Problem \eqref{eq:bvp}--\eqref{eq:signorinibc} admits the following
variational formulation. By introducing the convex set
\[
  K=\{v\in H^1(\Omega): v=0 \text{ on } \Gamma_D \text{ and } v\geq 0 
  \text{ on } \Gamma_S\},
\]
the function $y\in K$ satisfies the variational inequality
\begin{align}
  \int_\Omega (\nabla y\cdot \nabla (v-y))\geq 
  \int_{\Gamma_U} u (v-y)\quad \forall v\in K.
  \label{eq:VFsignorini}
\end{align}
The solution of \eqref{eq:VFsignorini} exists and is  unique, see
for instance \cite[Section II.2.1]{KinderlehrerStampaccia}.

Let us start with a first regularity result of this solution. In
particular it shows that the solution is continuous near the Signorini
boundary such that the definition of $\Gamma_C$ in \eqref{def:GammaC}
makes sense. To this end, introduce a domain 
\begin{align}\label{def:W}
  W\subset\mathbb{R}^2:\ \bar\Gamma_S\subset W,\ 
  \bar\Gamma_U\cap\bar W=\emptyset,
\end{align}
see the illustration in Figure \ref{fig:W}, and let 
\begin{figure}
  \begin{center}
    \includegraphics[scale=1]{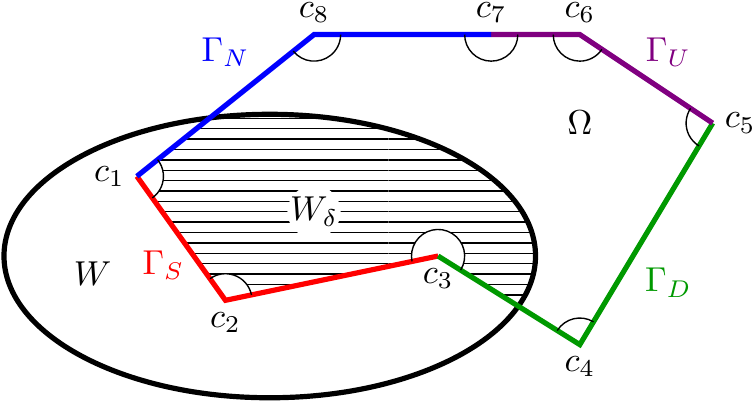} 
  \end{center}
  \caption{\label{fig:W}Illustration of the domains $\Omega$,
    $W$ and $W_\delta$.}
\end{figure}
\begin{align*}
  W_\delta:=(W\cap\Omega)\setminus\bigcup_{i=1}^n \bar B(c_i,\delta)
\end{align*} 
where the set $C=\{c_i\}_{i=1}^n$ of centers of balls with radius
$\delta>0$ is introduced in the introduction.
\pagebreak[3]

\begin{lemma}\label{t:regy}
  The solution $y\in K$ of \eqref{eq:VFsignorini} with $u\in
  H^{1/2}(\Gamma_U)$ satisfies
  \begin{equation}
    \label{reg:locale}
    y\in H^2(W_\delta)\cap C^{1,\lambda}(\bar W_\delta) 
    \quad\forall \lambda\in(0,\tfrac12)
  \end{equation}
  for any $\delta>0$.
\end{lemma} 

\begin{proof}
  We start with the proof of the property
  \begin{align}\label{eq:localreg}
    \forall x\in \bar W_\delta \quad \exists\, \varepsilon_x>0:\quad
    y\in H^2(W_\delta\cap B(x,\varepsilon_x))
  \end{align}
  using localization arguments. 
  \begin{itemize}
  \item If $x\in W_\delta$ we consider a ball $O_x=B(x,\varepsilon_x)$
    with $\varepsilon_x<\mathrm{dist}(x,\partial\Omega)$. The solution
    is harmonic in $O_x$ and hence even real analytic in $O_x$
    \cite[Theorem 1.7.1]{GLC}.
  \item For $x\in\Gamma_D\cap\bar W_\delta$ or $x\in\Gamma_N\cap\bar
    W_\delta$ we consider a neighborhood
    $O_x=B(x,\varepsilon_x)\cap\Omega$ with
    $\varepsilon_x<\mathrm{dist}(x,C)$.  Again, since the solution is
    harmonic in $O_x$ it is real analytic in $O_x$, \cite[Theorem
    2.7.1]{GLC}, i.\,e. near the smooth part of the Dirichlet or
    Neumann boundary.
  \item For the remaining case $x\in \Gamma_S\cap \bar W_\delta$ we
    fix a rotationally symmetric cut-off function $\eta\in
    \mathcal{D}(\R^2)$ such that $\eta=1$ in a neighborhood of $x$
    with a small support such that $\mathrm{supp}\,\eta\cap
    \Omega\subset W_{\delta/2}$, see the illustration in Figure
    \ref{fig:supp}.
    \begin{figure}
      \begin{center}
        \includegraphics[scale=1]{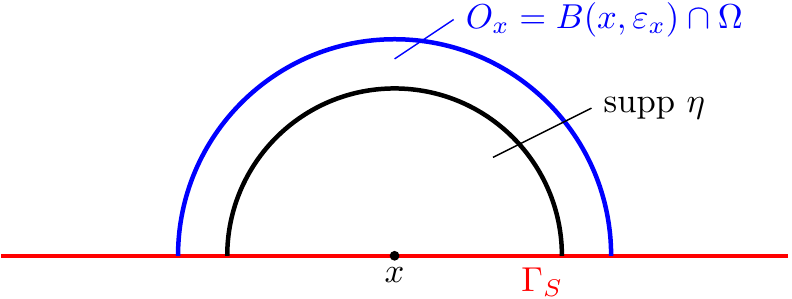}
      \end{center}
      \caption{\label{fig:supp}Illustration of $\mathrm{supp}\,\eta$
        and $O_x$.}
    \end{figure}
    Let now $O_x=B(x,\varepsilon_x)\cap\Omega$ with appropriately
    chosen $\varepsilon_x$ be a convex domain containing the support
    of $\eta$.  Then $v=\eta w-\eta^2 y+y$ with arbitrary \[w\in
    K_x:=\{z\in H^1(O_x): z\geq 0\ \text{on } \partial O_x\}\]
    satisfies
      $v=y=0$ on $\Gamma_D$ and
      $v=\eta w+(1-\eta^2)y\ge0$ on $\Gamma_S$
    since all factors are greater than or equal to zero; hence $v\in
    K$. Inserting $v$ into \eqref{eq:VFsignorini} gives
    \[
      \int_\Omega\nabla y\cdot\nabla(\eta(w-\eta y)) \ge
      \int_{\Gamma_U}u\eta(w-\eta y),
    \]
    and with $\eta\equiv0$ in $\Omega\setminus O_x$ we get
    \begin{align}\label{eq:newstar}
      \int_{O_x}\nabla y\cdot\nabla(\eta(w-\eta y)) \ge0
      \quad\forall x\in K_x.
    \end{align}
    Since $\partial_n\eta=0$ on $\partial O_x$ we get
    \begin{align*}
      \int_{O_x} &\nabla(\eta y)\cdot\nabla(w-\eta y) =
      \int_{O_x} (y\nabla\eta+\eta\nabla y)\cdot\nabla(w-\eta y) \\
      &=-\int_{O_x} \nabla\cdot(y\nabla\eta)\cdot\nabla(w-\eta y) +
      \int_{O_x} \eta\nabla y\cdot\nabla(w-\eta y) \\
      &=-\int_{O_x} \nabla\cdot(y\nabla\eta)(w-\eta y) +
      \int_{O_x} \nabla y\cdot\big(\nabla(\eta(w-\eta y))-\nabla\eta(w-\eta y)\big) \\
      &\ge-\int_{O_x} \left(\nabla\cdot(y\nabla\eta)+\nabla y\cdot\nabla\eta\right)(w-\eta y)
    \end{align*}
    due to \eqref{eq:newstar}. Hence $\eta y\in K_x$ can be seen as
    the unique solution of
    \begin{align*} 
      \int_{O_x} (\nabla (\eta y)\cdot \nabla (w-\eta y)+\eta y(w-\eta y))\geq 
      \int_{O_x} g_x (w-\eta y)\quad \forall w\in K_x
    \end{align*}
    with $g_x:=-\nabla\cdot(y\nabla\eta)-\nabla y\cdot\nabla\eta +\eta y \in
    L^2(O_x)$. Grisvard and Iooss showed that $\eta y\in H^2(O_x)$,
    see \cite[Corollary 3.2]{Grisvard-Iooss:76}.
  \end{itemize}
  Altogether the property \eqref{eq:localreg} is proved.

  The balls $B(x,\varepsilon_x)$ form an open covering of $\bar
  W_\delta$, hence there exists a finite covering, i.\,e.,
  \begin{align*}
    \exists\,x_j, \ j=1,\ldots,J: \quad
    \bar W_\delta\subset \bigcup_{j=1}^J B(x_j,\varepsilon_{x_j}).
  \end{align*}
  We conclude that 
  \[
    y\in H^2(W_\delta)\subset W^{1,p}(W_\delta)\quad\forall p\in [1,\infty).
  \]

  With the same procedure as above we can now prove that $y\in
  C^{1,\lambda}(W_\delta\cap B(x,\varepsilon_x))$ for all $x \in \bar
  W_{2\delta}$. The point is that now $g_x\in L^p(O_x)$ for all
  $p<\infty$ such that we can use a theorem from Khodja and Moussaoui,
  \cite[Theorem 2]{MoussaouiKhodja1992} (see also
  \cite{Uraltseva:87}), to deduce that $\eta_x y$ belongs to
  $C^{1,\lambda}(O_x)$ for all $\lambda\in(0,\frac12)$. As above we
  conclude $y\in C^{1,\lambda}(W_{2\delta})$. Since $\delta>0$ was
  arbitrary we are done.
\end{proof}

The following lemma is inspired from \cite[\S 6]{kinderlehrer:81Pisa},
see also \cite[Lemma III.1.3]{MoussaouiKhodja1992}.

\begin{lemma}\label{l:zero}
  Denote by $\partial_ty$ and $\partial_ny$ the tangential and normal
  derivatives along the boundary. Then the equality
  \begin{align}
    \partial_ty \, \partial_ny =0 \ \text{on } \Gamma_S\setminus C.
    \label{eq:prod=0}
  \end{align}
  holds.
\end{lemma}

\begin{remark}\label{rem:zero}
This result extends even to $\Gamma_D$ and $\Gamma_N$ since
$\partial_ty=0$ on $\Gamma_D$ and $\partial_ny=0$ a.e. on $\Gamma_N$.
\end{remark}

\begin{proof}
  Introduce the compact set
  \[
  \Gamma_{S,\varrho}:=\Gamma_S\setminus\bigcup_{i=1}^nB(c_i,\varrho)
  \]
  for some $\varrho>0$.
  Then according to Theorem \ref{t:regy}, $\partial_n y$ is continuous on
  $\Gamma_{S,\varrho}$, hence we can introduce the sets
  \begin{align*}
    \Gamma_{S,\varrho}^+&=\{x\in \Gamma_{S,\varrho}\ | \ \partial_n
    y(x)>0\},& \Gamma_{S,\varrho}^0&=\{x\in \Gamma_{S,\varrho}\ | \
    \partial_n y(x)=0\},
  \end{align*}
  and notice that
$
  \Gamma_{S,\varrho}=\Gamma_{S,\varrho}^+\cup \Gamma_{S,\varrho}^0.
$ 
  At this stage, we distinguish two cases: 
  \begin{enumerate}
  \item If $x\in
  \Gamma_{S,\delta}^0$, we have $\partial_n y(x)=0$ and hence
  \begin{align}
    {\partial_t}y (x) \, \partial_ny(x) =0.
    \label{eq:prod=0ponctual}
  \end{align}
  \item In the other case, $x\in \Gamma_{S,\delta}^+$, we have $y(x)=0$
  due to the Signorini conditions. Observe that the continuity of
  $\partial_n y$ implies that $\Gamma_{S,\varrho}^+$ is an open subset
  of $\Gamma_{S,\varrho}$. Hence, if $x\in \Gamma_{S,\delta}^+$, then
  $y=0$ holds in a neighborhood of $x$, and the tangential derivative
  is also zero in this neighborhood and consequently $\partial_ty
  (x)=0$, which shows that \eqref{eq:prod=0ponctual} also holds in
  that case.
  \end{enumerate}
  We have just shown that \eqref{eq:prod=0ponctual} is valid for all
  $x\in \Gamma_{S,\varrho}$ and letting $\varrho$ tend to zero we find
  \eqref{eq:prod=0}.
\end{proof}

We prove now the main result of this section, namely the
characterization of the coincidence set $\Gamma_C$, see
\eqref{def:GammaC}.
For that purpose, we adapt the method
of Kinderlehrer in \cite[\S 6]{kinderlehrer:81Pisa} who treated the
case of the elasticity system.

\begin{theorem}\label{thm:coincideset1}
  Let $y\in K$ be the unique solution of \eqref{eq:VFsignorini},
  then the coincidence set $\Gamma_C$ is the union of a finite
  numbers of intervals and finitely many isolated points.
\end{theorem} 

\begin{proof}
  We localize the problem by considering a finite covering of
  $\Gamma_S$.
  Introduce a finite number of open balls $B(c_i,\varrho_i)$, $i\in
  J$. The index set $J$ is chosen such that $J\supset \{i\in C :
  x_i\in \bar\Gamma_S\}$ and the radii $\varrho_i>0$ are chosen such
  that $\Gamma_S\subset\bigcup_{i\in J}B(c_i,\varrho_i)$ and
  $c_j\not\in B(c_i,\varrho_i)$ for $i\ne j$, see Figure
  \ref{fig:domain}. Note that the index set may contain further points
  $c_i\in\Gamma_S\setminus C$.
  \begin{figure}
    \begin{center}
      \includegraphics[scale=1]{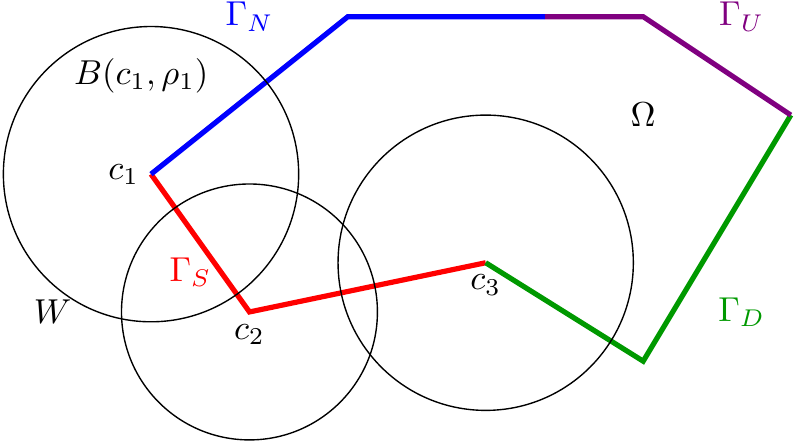}
    \end{center}
    \caption{\label{fig:domain}Illustration of the domain, the
      boundary conditions and the covering.}
  \end{figure}

  We consider now any ball $B(c_i,\varrho_i)$  and
  omit the index $i$ for better readability.  Introduce a local polar
  coordinate system $(r,\theta)$ centered in $c$, such that
  \[
    O_\alpha:=B(c,\varrho)\cap\Omega =
    \{(r\cos\theta,r\sin\theta)\in\mathbb{R}^2:
    0<r<\varrho,0<\theta<\alpha\} 
  \] 
  where $\alpha$ is the angle of the domain at $c$.
  Consider now the situation where
  \[
    \Gamma_0:=\{(r\cos\theta,r\sin\theta)\in\mathbb{R}^2:0<r<\varrho,\theta=0\}
    \subset\Gamma_S.
  \]
  The other leg
  $\Gamma_\alpha:=\{(r\cos\theta,r\sin\theta)\in\mathbb{R}^2:
  0<r<\varrho, \theta=\alpha\}$ may be contained in $\Gamma_D$,
  $\Gamma_N$ or $\Gamma_S$ but not in $\Gamma_U$ because of
  $\bar\Gamma_S\cap\bar\Gamma_U=\emptyset$, see \eqref{boundaryparts}.
  Note that the situation where $\Gamma_\alpha\subset\Gamma_S$ and
  $\Gamma_0\not\subset\Gamma_S$ can be treated in a similar way.

  The function $y$ satisfies
  \begin{align*}
    -\Delta y &=0 \quad\text{in } O_\alpha, \\
    y\geq 0,\ \partial_n y\geq 0, \ y \partial_n
    y&=0 \quad\text{on }\Gamma_S\cap B(c,\varrho).
  \end{align*}

  Now regarding $y$ as a function of the complex variable
  $z=x_1+\imath x_2$, we define the function
  \[
    w(z)=\partial_2 y (z)+\imath \partial_1y(z),
  \]
  defined in $O_\alpha$ (now considered as a subset of $\mathbb{C}$), see the
  illustration in Figure \ref{fig:trafo}.
  \begin{figure}
    \begin{center}
      \includegraphics[scale=1]{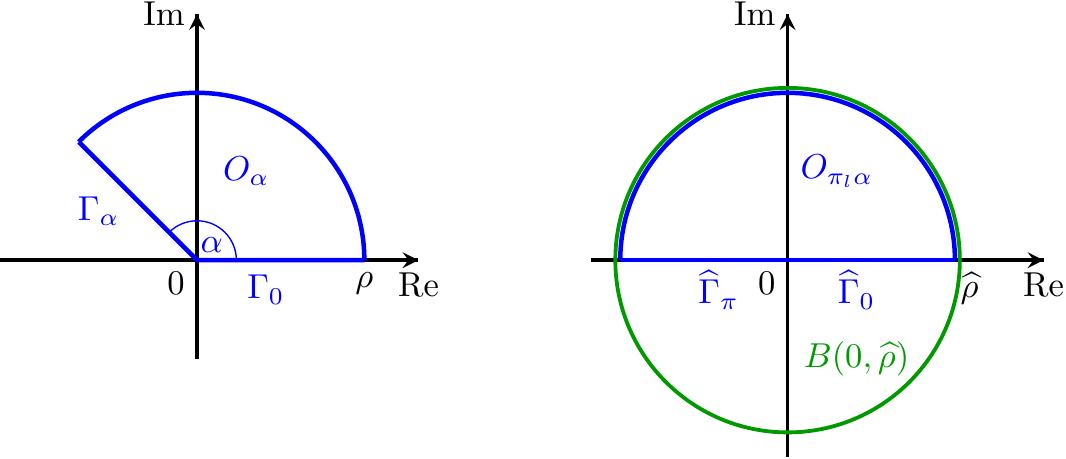}
    \end{center}
    \caption{\label{fig:trafo}Illustration of transformation.}
  \end{figure}
  As $y$ is harmonic in~$O_\alpha$ and belongs to
  $C^2(O_\alpha)$, the function $w$ is analytic in $O_\alpha$,
  \cite[p. 41]{conway}.  Furthermore, we introduce the biconformal
  mapping
  \[
    h: z\mapsto\hat z= z^{\pi/\alpha}, \quad O_\alpha\to O_{\pi,\alpha}:=
    \{z=r\mathrm{e}^{\imath\theta}\in\mathbb{C}:
    0<r<\varrho^{\pi/\alpha}, 0<\theta<\pi\},
  \]
  and denote $\hat\Gamma_0:=h(\Gamma_0)$ and
  $\hat\Gamma_\pi:=h(\Gamma_\alpha)$. Note the simple rule
  \[
    h:r\mathrm{e}^{\imath\theta}\mapsto\hat
    r\mathrm{e}^{\imath\hat\theta}\quad\text{with } \hat
  r=r^{\pi/\alpha} \text{ and } \hat\theta=\frac{\theta\pi}{\alpha}.
  \]

  Let us analyze now the function
  \[
    \hat y(\hat z):=y(z).
  \]
  Since a conformal mapping preserves the Laplace operator (up to a
  factor) and since the normal derivative is up to a factor again
  the $\theta$-derivative we get
  \begin{align*}
    -\Delta\hat y &=0 \quad\text{in }O_{\pi,\alpha}, \\
    \hat y\ge0, \quad \hat\partial_n\hat y\ge0,\quad
    \hat y \hat\partial_n\hat y&=0 \quad\text{on }\hat\Gamma_0 ,
  \end{align*}
  and the appropriate Dirichlet, Neumann or Signorini boundary
  condition on $\hat\Gamma_\pi$. Moreover, we can compute
  \begin{align*}
    \partial_r y&=
    \partial_r\hat r\,\partial_{\hat r}\hat y=
    \tfrac\pi\alpha r^{\pi/\alpha-1} \partial_{\hat r}\hat y , &
    \partial_\theta y&=
    \partial_\theta\hat\theta\,\partial_{\hat\theta}\hat y=
    \tfrac\pi\alpha\partial_{\hat\theta}\hat y,
  \end{align*}
  such that
  \begin{align*}
    \int_{O_{\pi,\alpha}} |\hat\nabla\hat y|^2 &=
    \int_{O_{\pi,\alpha}} \left( |\partial_{\hat r}\hat y|^2+
    |\hat r^{-1}\partial_{\hat\theta}\hat y|^2
    \right) \,\hat r\mathrm{d}\hat r\mathrm{d}\hat\theta \\ &=
    \int_{O_\alpha} \left( 
    (\tfrac\pi\alpha)^{-2} r^{2(1-\pi/\alpha)} |\partial_r y|^2 +
    r^{-2\pi/\alpha}(\tfrac\pi\alpha)^{-2}|\partial_\theta y|^2 
    \right) \,r^{\pi/\alpha}\,
    \tfrac\pi\alpha r^{\pi/\alpha-1}\mathrm{d}r\,
    \tfrac\pi\alpha\mathrm{d}\theta\\ &=
    \int_{O_\alpha} \left( 
    |\partial_r y|^2 + |r^{-1}\partial_\theta y|^2 
    \right) \,r\mathrm{d}r\mathrm{d}\theta =
    \int_{O_\alpha} |\nabla y|^2,
  \end{align*}
  i.\,e., for the function
  \[
    \hat w(\hat z):=\partial_1\hat y+\imath\partial_2\hat y
  \]
  the relation
  \begin{align*}
    |\hat w(\hat z)^2|&=|\hat w(\hat z)|^2=
    |\hat\nabla\hat y|^2\in L^1(O_{\pi,\alpha})
  \end{align*}
  holds.

  Lemma \ref{l:zero} and Remark \ref{rem:zero} imply that
  \[
    \Im \left(w(z)\right)^2=0 \hbox{ on } 
    (-\hat\varrho,\hat\varrho)\setminus\{0\},
  \]
  with $\hat\varrho=\varrho^{\pi/\alpha}$.  Consequently on
  $U:=B(0,\hat\varrho)\setminus\{0\}$, we define the function
  \[
  F(z)=
  \begin{cases}
    w(z)^2 &\hbox{ if } \Im z\leq 0,\\[0.5ex]
    \overline{w(\bar z)^2} &\hbox{ if } \Im z> 0,
  \end{cases}
  \]
  which is analytic in $U$ by the Schwarz reflection principle, see
  \cite[\S IX.1]{conway}. Hence $F$ is meromorphic in $U$. As
  additionally $F$ belongs to $L^1(U)$ we conclude that $F$ admits the
  Laurent expansion
  \[
  F(z)=\frac{c}{z}+F_H(z),
  \]
  with $c\in \mathbb{C}$ and a function $F_H$ which is an analytic in
  $B(0,\hat\varrho)$. (Terms $z^{-j}$ with $j>1$ are not in $L^1(U)$.)
  This implies that the function
  \[
  \Phi(z)=z F(z)
  \]
  is holomorphic on $B(0,\hat\varrho)$. Therefore $\Phi$ has a finite number of
  zeroes on $\hat\Gamma:=\{z\in\bar U:\Im z=0\}$ if $\Phi$ is not
  identically equal to 0.

  Let us analyze two cases:
  \begin{enumerate}
  \item If $\Phi$ is identically equal to 0, then by
    \eqref{eq:prod=0} we get
    \[
    w(z)^2=(\partial_2y)^2-(\partial_1 y)^2= 0 \hbox{ on }
    \hat\Gamma\setminus\{0\},
    \]
    which again by \eqref{eq:prod=0} implies
    \[
    \partial_2y=\partial_1 y=0 \hbox{ on } \hat\Gamma\setminus\{0\}.
    \]
    Consequently $y$ is constant on $\hat\Gamma$, so either this
    constant is zero and $\hat\Gamma_C:=\{z\in\hat\Gamma:\hat y(z)=0\}
    = \hat\Gamma$, or this constant is different from zero and
    $\hat\Gamma_C= \emptyset$.
    
  \item If $\Phi$ is not identically equal to 0, the sets $\{z\in
    \hat\Gamma: \Phi(z)>0\}$ and $\{z\in \hat\Gamma: \Phi(z)<0\}$
    are unions of a finite number of intervals $I$.  We are
    looking for the behavior of $y$ on any of these intervals $I$.
    Depending on the sign of $z\in\hat\Gamma$ we find that
    $\Phi(z)=zF(z)$ is positive or negative in $I$, hence $(\partial_2
    y)^2-(\partial_1 y)^2$ does not change sign in $I$, and
    moreover $(\partial_2 y)^2>0$ or $(\partial_1 y)^2>0$ in
    $I$. If $(\partial_2 y(z))^2>0$ then we get by the Signorini
    condition that $y\equiv 0$ in $I$. If $(\partial_1 y(z))^2>0$ then
    the function $y$ is nowhere constant in $I$, hence $y$ has no or a finite number
    of zeros in $I$, and we get by the Signorini condition that
    $\partial_n y=0$ a.e. in $I$.
  \end{enumerate}
  In conclusion, in this case $\hat\Gamma_C$ is the union of a finite
  number of intervals plus eventually a finite number of points. Since
  the mapping $h$ is continuous this result is valid also for~$\Gamma_C$.
\end{proof}

\begin{remark}\label{rem:f}
  Let us finish this section with a discussion of our assumption that
  we assumed a homogeneous differential equation in \eqref{eq:bvp} and
  a homogeneous gap function in \eqref{eq:signorinibc}.
  \begin{itemize}
  \item The assumption of a homogeneous differential equation in
    \eqref{eq:bvp} was made to simplify the discussion. For Lemma
    \ref{t:regy} a right hand side $f\in L^\infty(\Omega)$ could be
    admitted. Recall also the introduction of the domain $W$ in
    \eqref{def:W}. The whole analysis is untouched if the equation is
    homogeneous in a neighborhood of $\Gamma_S$ only since then the
    set $W$ could be defined accordingly.
  \item In particular cases the solution of non-homogeneous problem
    could be homogenized. Assume that the differential equation in
    \eqref{eq:bvp} is replaced by $-\Delta y=f$ and the gap condition
    in \eqref{eq:signorinibc} is replaced by $y\ge\psi$. If $f$ and
    $\psi$ are such that there exists a function $y_{f,\psi}$ such that
    \begin{align*}
      -\Delta y_{f,\psi} &= f \quad\text{in }\Omega, \\
      y_{f,\psi} &= 0 \quad\text{on }\Gamma_D,\\
      \partial_n y_{f,\psi}&=0 \quad\text{on }\Gamma_N\cup\Gamma_U,\\
      y_{f,\psi}=\psi,\ \partial_n y_{f,\psi}&=0 \quad\text{on }\Gamma_S,
    \end{align*}
    then $y-y_{f,\psi}$ satisfies our assumptions. Of course this
    problem is overdetermined such that the existence of a solution
    cannot be expected for any $f$ and $\psi$. But examples can be
    constructed by choosing a function
    \[
      y_*\in\{v\in H^2(\Omega):v=0 \text{ on }\Gamma_D, 
      \partial_n v=0 \text{ on }\Gamma_N\cup\Gamma_U\cup\Gamma_S\}
      \supset H^2_0(\Omega)
    \]
    and defining $f=-\Delta y_*$ and $\psi=y_*|_{\Gamma_S}$.
  \end{itemize}
\end{remark}

\begin{example}
 Christof and Haubner investigated in \cite{ChristofHaubner18} a
  square domain and the case $\Gamma=\Gamma_S$. In the case of a
  homogeneous differential equation, Condition (A) in this paper is
  now proven in Theorem \ref{thm:coincideset1}, namely that the 
  relative boundary of $\Gamma_C$ has one-dimensional Hausdorff
  measure zero and the relative interior of $\Gamma_C$ consists of at
  most finitely many connected components.
\end{example}

\section{\label{sec:reg}Regularity of the solution}

We formulate now a regularity result in the spirit of Theorem 2.3 of the paper
\cite{ChristofHaubner18} by Christof and Haubner where the regular part of the
solution is considered in $W^{2,p}$, $p>2$, $p\not=4$. But we like to
note that the regular part could also be smoother; the prize is that
possibly more singular terms have to be included and the datum $u$ at
the Neumann boundary must be sufficiently regular.

\begin{theorem}\label{t:reg}
  Let $y$ be the solution of problem
  \eqref{eq:bvp}--\eqref{eq:signorinibc}. Recall the set
  $\{c_i\}_{i=1}^n$ of critical points and the interior angles
  $\alpha_i$. Recall also that there are points
  $\{c_i\}_{i=n+1}^m\subset\Gamma_S\setminus C$ of unknown location
  which are the endpoints of the intervals in the coincidence set and in that case, set $\alpha_i=\pi$.
  Furthermore denote by $(r_i,\theta_i)$ local polar coordinates at
  all these points.  

  Let $p>2$, $p\not\in P$, where the finite set $P$ of exceptional
  values is a subset of the countable set 
  \[
    \left\{ \frac{2}{2-\frac{k\pi}{2\alpha_i}},
    \ k\in\mathbb{N},\ i=1,\ldots,m\right\}.
  \]
  \pagebreak[3]

  \noindent Assume that $u\in W^{1-1/p,p}(\Gamma_U)$ satisfies the compatibility
  condition $u(c_i)=0$ 
  \begin{itemize}
  \item if $c_i\in \bar\Gamma_D\cap\bar\Gamma_U$ and
    $\alpha_i=\frac12\pi$ or $\alpha_i=\frac32\pi$ or
  \item if $c_i\in \bar\Gamma_N\cap\bar\Gamma_U$ and $\alpha_i=\pi$. 
  \end{itemize}
  Then there is a representation of $y$
  \begin{align*}
    y=y_R+\sum_{i=1}^n\sum_{j: 0<\lambda_{i,j}<2-2/p\atop \lambda_{i,j}\ne 1} d_{i,j}
    r_i^{\lambda_{i,j}} \Phi_{i,j}(\theta_i) +
    \sum_{i=n+1}^m d_i r_i^{3/2}\Phi_{i}(\theta_i)
  \end{align*}
  with $y_R\in W^{2,p}(\Omega)$, coefficients $d_{i,j}$ and $d_i$, smooth
  functions $\Phi_{i,j}$ and $\Phi_{i}$, and exponents
  \[
    \lambda_{i,j}=
    \begin{cases}
      j\pi/\alpha_i & \text{if D-D or N-N or U-U or U-N conditions near }c_i,\quad j\ge1, \\
      (j-\frac12)\pi/\alpha_i  & \text{if D-N or D-U conditions near }c_i,\quad j\ge1, \\
      j\pi/(2\alpha_i)  & \text{if S-S conditions near }c_i,\quad j\ge2, \\
      j\pi/(2\alpha_i)  & \text{if S-D or S-N conditions near }c_i,\quad j\ge1,
    \end{cases}
  \]
  where D-N means that one boundary edge at $c_i$ is contained in
  $\Gamma_D$ and the other in $\Gamma_N$, and so on.
\end{theorem}

\begin{remark}
  The compatibility conditions could be omitted, but then a
  singularity of the form $r(\Phi_1(\theta)+\log r\Phi_2(\theta))$
  with $\Phi_i$ smooth has to be added, see \cite[p.
  263]{grisvard:85a}.
\end{remark}

\begin{proof}
  Since we have a finite number of critical boundary points $c_i$ due
  to Theorem \ref{thm:coincideset1} we can treat them separately and
  use classical theory as described for instance in \cite[Corollary
  4.4.4.14]{grisvard:85a}. Let us discuss shortly the situation near the
  Signorini boundary.

  For $i\le n$ and $c_i\in\bar\Gamma_S\cap\bar\Gamma_D$ or
  $c_i\in\bar\Gamma_S\cap\bar\Gamma_N$ we do not know whether a
  Dirichlet or Neumann boundary condition occurs on $\Gamma_S$ near
  $c_i$. Therefore we consider the worst situation of mixed boundary
  conditions.

  In the remaining cases some singularities disappear at
  $c_i\in\bar\Gamma_S$:
  \begin{enumerate}
  \item For $i=n+1,\ldots,m$ the leading singularity is $r_i^{3/2}$
    since the term $r_i^{1/2}$ is not in $H^2(\Omega)$, compare the
    result in Lemma \ref{t:reg}, and see the discussion in e.\,g.
    \cite{ChristofHaubner18,MoussaouiKhodja1992}.
  \item For $i\le n$ and $c_i\in\Gamma_S$ the worst situation could be
    mixed. Let us consider the case that Dirichlet conditions is valid
    for $\theta_i=0$. Then we have in the vicinity of $c_i$
    \begin{equation}\label{expwith small remainder}
      y=y_r+
      dr_i^{\pi/(2\alpha_i)}\sin\left(\frac{\pi\theta_i}{2\alpha_i}\right).
    \end{equation}
    We show now $y_r=o(r_i^{\pi/(2\alpha_i)})$ such that this
    term is neglectable sufficiently close to $c_i$.  Indeed from
    \cite[Corollary 4.4.4.14]{grisvard:85a} near $c_i$, we have
    \begin{equation}\label{fullexp}
      y_r=y_R+ \sum_{j\in \mathbb{N}\setminus\{0\}:
      0<(j+1/2)\pi/\alpha_i <2-2/p\atop (j+1/2)\pi/\alpha_i \ne 1} 
      d_{i,j}r_i^{(j+1/2)\pi/\alpha_i} \sin((j+1/2)\pi \theta_i/\alpha_i),
    \end{equation}
    with $y_R\in W^{2,p}(\Omega\cap B(c_i, \rho))$ for $\rho$ small
    enough and $d_{i,j}\in \mathbb{R}$.  Consequently, near $c_i$,
    \begin{equation}\label{bcyR}
      y_{R}(r_i,0)=0,\quad \frac{\partial y_{R}}{\partial \theta_i}(r_i,\alpha_i)=0
     \quad \forall r_i<\rho.
    \end{equation}
    Notice that the Sobolev embedding theorem guarantees that
    \begin{equation}\label{eq:star}
      y_R\in C^{1,\beta}(\bar \Omega\cap \bar B(c_i, \rho))\ \text{with }\beta=1-2/p.
    \end{equation}
    
    We now notice that the second term in the sum in
    \eqref{fullexp} (if any) is trivially $o(r_i^{\pi/(2\alpha_i)})$,
    hence it remains to check the same behavior for $y_R$. We note
      that $\nabla y_R(c_i)=0$ except in the cases
      $\alpha_i=\frac12\pi$ or $\alpha_i=\frac32\pi$, and that
      $r_i^{(j+1/2)\pi/\alpha_i}$ is smooth when
      $\alpha_i=\frac12\pi$. For that purpose, we distinguish three
    cases.
    \begin{enumerate}
    \item If $\pi/(2\alpha_i)<1$, then by Taylor's theorem (and since
      $ y_R(c_i)=0$), we have
      \[
      y_R(x)=\nabla y_R(c_i)\cdot (x-c_i)+o(r_i),
      \]
      which yields $y_R(x)=O(r_i)=o(r_i^{\pi/(2\alpha_i)})$ as
      $\pi/(2\alpha_i)<1$.
    \item If $\pi/(2\alpha_i)=1$, then from \cite[Corollary
      4.4.4.14]{grisvard:85a}, we directly have
      \[
      y=y_R\in W^{2,p}(\Omega\cap B(c_i, \rho)),
      \]
      in other words the singular part is zero.
    \item If $\pi/(2\alpha_i)>1$, then owing to \eqref{bcyR} and the
      regularity of $y_R$, we actually have
      \[
      \nabla y_R(c_i)=0,
      \]
      hence by Taylor's expansion with an integral remainder, we have
      \[
      y_R(x)=\int_0^1 \nabla y_R(c_i+t(x-c_i)) (x-c_i)\,dt, \forall
      |x-c_i|<\rho.
      \]
      Therefore as $|\nabla y_R(c_i+t(x-c_i))| =|\nabla
      y_R(c_i+t(x-c_i))-\nabla y_R(c_i)|=O((tr_i)^{\beta})$ due to
      \eqref{eq:star}, one deduces that
      \[
      | y_R(x)|=O(r_i^{\beta+1})=o(r_i^{\pi/(2\alpha_i)}),
      \]
      as $\pi/(2\alpha_i)<\beta+1$. (In the case
        $\pi/(2\alpha_i)>\beta+1=2-2/p$ the solution $y$ is
        $W^{2,p}$-regular in the vicinity of $c_i$. Equality is
        excluded by assumption.)
    \end{enumerate}
    
    Coming back to \eqref{expwith small remainder}, for $\theta_i=0$ we get
    $\partial_ny=\partial_ny_r-d r_i^{\pi/(2\alpha_i)-1}$, hence
    $d\le0$ in order to satisfy the Signorini condition
    $\partial_ny\ge0$. For $\theta_i=\alpha_i$ we get
    $y=y_r+dr_i^{\pi/(2\alpha_i)}$, hence $d\ge0$ in order to satisfy
    the Signorini condition $y\ge0$. So we can have only $d=0$.
  \end{enumerate}
  Since all cases are treated the proof is complete.
\end{proof}

\begin{example}
  Let us shortly discuss the L-domain; that is a hexahedron with one
  interior angle $\alpha=\frac32\pi$ and all others being of size
  $\frac12\pi$. The leading singular term near the non-convex corner
  is of type $r^\lambda$ with $\lambda=\frac23$ if Signorini
  conditions are given at both legs of this angle, but with
  $\lambda=\frac13$ if a Signorini condition is given on one leg only,
  and a Dirichlet or Neumann condition at the other leg. These terms
  are in $H^s(\Omega)$ for $s<1+\lambda$ or in a suitable weighted
  Sobolev space. The set of exception values for $p$ is $P=\{3,6\}$.
\end{example}

\paragraph{Acknowledgment}
The authors thank Constantin Christof for pointing to an incorrect
argument in a previous version of the paper. The authors thank also
Christof Haubner for preparing the illustrations.

\bibliographystyle{abbrv}\bibliography{Signorini}

\end{document}